\documentclass[12pt,a4paper]{amsart}

\usepackage{latexsym}
\usepackage{amsfonts,amsmath,amssymb,indentfirst}
\usepackage[english]{babel}
\usepackage[all]{xy}
\usepackage[pdftex]{hyperref}

\newcommand{\bdism}{\begin{displaymath}}
\newcommand{\edism}{\end{displaymath}}

\newcommand{\pp}{\mathbb{P}}
\newcommand{\oo}{\mathcal{O}}

\newcommand{\K}{\kappa_\sigma}

\DeclareMathOperator{\vol}{vol}

\DeclareMathOperator{\supp}{Supp}

\DeclareMathOperator{\bs}{Bs}
\DeclareMathOperator{\B}{\mathbf{B}}

\newtheorem{theorem}{Theorem}[section]
\newtheorem{proposition}[theorem]{Proposition}

\newtheorem{lemma}[theorem]{Lemma}

\newtheorem{definition}[theorem]{Definition}

\newtheorem{question}[theorem]{Question}

\address{Department of Mathematics, Princeton University, Princeton NJ 08540,
USA} \email{dicerbo@math.princeton.edu}

\author{\scshape Gabriele Di Cerbo}
\title{\bf Asymptotic growth of global sections on open varieties}
\begin{document}
\pagestyle{headings}

\begin{abstract}
Let $X$ be a projective variety and let $E$ be a reduced divisor. We study the asymptotic growth of the dimension of the space of global sections of powers of a divisor $D$ on $X\backslash E$. We show that it is always bounded by a polynomial of degree $\dim(X)$, if finite. Furthermore, when $D$ is big, we characterize the finiteness of the cohomology groups in question.  This answers a question of Zariski and Koll\'ar. 
\end{abstract}
\maketitle

\tableofcontents
\markboth{GABRIELE DI CERBO}{\textbf{ASYMPTOTIC GROWTH OF GLOBAL SECTIONS}}

\section{Introduction}

\pagenumbering{arabic}

Hilbert's 14th problem on the finite generation of invariants was reinterpreted by Zariski in \cite{Zar} as a question on the finite generation of the ring of global sections on a quasi-projective variety.
The quasi-affine version, studied by Zariski, asks for the finite generation of
$H^0\bigl(U, \oo_U\bigr)$ for any quasi-affine variety $U$. The quasi-projective version asks for the finite generation of $\bigoplus_{m=0}^{\infty}  H^0\bigl(U, \oo_U(mD)\bigr)$ for a  Cartier divisor $D$. The latter is of course only possible if
the individual groups $H^0\bigl(U, \oo_U(mD)\bigr)$ are finite dimensional for every $m$.
Nagata's examples show that finite generation frequently fails in both cases.
As a weakening of finite generation, Koll\'ar asked whether the groups
$H^0\bigl(U, \oo_U(mD)\bigr)$ always have polynomial growth. The aim of this note is to answer this affirmatively.

\begin{theorem}\label{main}
Let $X$ be a normal projective variety of dimension $n$ defined over an algebraically closed field and let $E$ be a reduced divisor. Set $U:=X\backslash E$. Let $D$ be a Cartier divisor on $X$ such that $H^{0}(U,\oo_{U}(m D|_{U}))$ is finite dimensional for any $m>0$. Then there exists a constant $C>0$ such that
\bdism
h^{0}(U,\oo_{U}(m D|_{U}))\leq C m^{n},
\edism 
for any $m$.
\end{theorem}

Moreover, we characterize when the finite dimensionality in Theorem \ref{main} is satisfied, assuming that $D$ is big. We denote by $D = P_\sigma(D) + N_\sigma (D)$ the Zariski decomposition of a pseudo-effective divisor as defined in \cite{Nak}.

\begin{theorem}\label{bigtheo}
$X$, $D$, $E$ and $U$ as above. Furthermore, assume that $D$ is big. Then $H^{0}(U,\oo_{U}(m D|_{U}))$ is finite dimensional for any $m>0$ if and only if there exists $a\geq 0$ such that $E\leq N_\sigma (D+aE)$.
\end{theorem} 

We can think of elements in $H^{0}(U,\oo_{U}(m D|_{U}))$ as sections of $\oo_X(mD)$ with pole of arbitrary degree along $E$. Theorem \ref{bigtheo} implies that the order of the pole along $E$ is at most linear in $m$.

As the proof of the projective version does not generalize to our setting, we will use completely different techniques. Most of the proofs here are based on the results contained in \cite{FKL}. Moreover, we will use them to prove some useful results relating different notions of numerical dimension of divisors in our special case.

In the final section of the paper, we discuss in more details the case of surfaces and we prove an analogous of Theorem \ref{bigtheo} in the more general case of a pseudo-effective divisor $D$.
\vspace{0.5cm}

\noindent\textbf{Acknowledgements}. 
This paper was inspired by a question of J\'anos Koll\'ar. I would like to thank him for his support and many valuable conversations. I also thank John Lesieutre and Roberto Svaldi for helpful conversations.  

Partial financial support was provided by the National Science
Foundation under grants numbers DMS-1702358 and DMS-1440140 while the author was in residence at
the Mathematical Sciences Research Institute in Berkeley, California, during the Spring 2019 semester.

\section{Preliminaries}

For the convenience of the reader, let us recall a classical result that is usually obtained as a consequence of the asymptotic Riemann-Roch theorem and Serre vanishing. See \cite{Laz1} for more details and a complete proof of the following statement.  

\begin{theorem}\label{old}
Let $X$ be a projective variety of dimension $n$, and let $D$ and $A$ be Cartier divisors on $X$. Then there is a constant $C>0$ such that 
\[
h^0 (\oo_X(mD+A)) \leq C m^n.
\] 
\end{theorem}

This result has several important consequences. For example, it allows us to define a number of numerical invariants associated to a divisor $D$. For us, one of the most important is the numerical dimension.

\begin{definition}
Let $D$ be a Cartier divisor. The numerical dimension $\K(D)$  of $D$ is defined as 
\[
\K(D) = \max_k \left\{ \limsup_{m\to\infty} \frac{ h^0 (\oo_X(mD+A))} {m^k} >0 \right\},
\]
where $A$ is a sufficiently ample divisor.
\end{definition}

It is easy to see that $\K(D)$ does not depend on the choice of $A$ and that $\kappa(D)\leq \K(D) \leq n$, where $\kappa(D)$ is the Kodaira dimension of $D$. We say that $D$ is big if it has maximal numerical dimension. Moreover, we define its volume as 
\[
\vol(D) = \limsup_{m\to\infty} \frac{h^0 (X,\oo_X(mD))} {m^n/n!}.
\]

Note that $\vol(D)>0$ if and only if $\K(D)$ is maximal.

\subsection*{Augmented base locus} Associated to a divisor we also have few loci that describe its positivity.

\begin{definition}
The stable base locus of a divisor $D$ is 
\bdism
\B(D):=\bigcap_{m\geq 1}\bs(|mD|), 
\edism
where $\bs(|D|)$ is the base locus of $D$. \\ 
The augmented base locus of a divisor $D$ is the Zariski-closed set
\bdism \B_{+}(D):=\B(D-\epsilon A), 
\edism
for any ample $A$ and sufficiently small $\epsilon>0$.
\end{definition}

Note that $\B_{+}(D)\neq X$ if and only if $D$
is big and $\B_{+}(D)=\emptyset$ if and only if $D$ is ample. 

Following \cite{ELM}, we define the diminished base locus as 
\[
\B_-(D) := \bigcup_{\epsilon>0} \B_+(D + \epsilon A).
\]

The diminished base locus is empty if and only if $D$ is nef and it is closely related to the negative part of the Zariski decomposition of a pseudo-effective divisor. 

We can characterize numerically the augmented base locus using the restricted volume. 

\begin{definition}
Let $D$ be a Cartier divisor and let $V$ be a subvariety of $X$. Then the restricted volume $\vol_{X|V}(D)$ is defined to be
\[
\vol_{X|V}(D) = \limsup_{m\to\infty} \frac{\dim Im \left( H^0 (X, \oo_X(mD)) \rightarrow  H^0 (V, \oo_V(mD)) \right)}{m^{\dim V} / \dim V!}
\]
\end{definition}

In \cite{ELMNP}, the authors showed that the vanishing of the restricted volume guarantees that the subvariety lies in the augmented base locus.

\begin{theorem}\label{augmented}
Let $V \subset X$ be a subvariety and let $D$ be a Cartier divisor such that that $V\nsubseteq \B_{+}(D)$. Then $\vol_{X|V}(D) > 0$.
\end{theorem}

Moreover, if $V$ is a divisor, then the restricted volume can be interpreted as a directional derivative of the volume.

\begin{theorem}[\cite{LM}, Corollary 4.24]\label{okounkov}
Let $D$ be a big Cartier divisor. Suppose $E\nsubseteq \B_+ (D+aE)$ for some $a > 0$. Then 
 \[
\vol(D+aE)-\vol(D)=n\int_{0}^{a}\vol_{X|E}(D+tE)dt.
\]
\end{theorem}

\subsection*{Zariski decomposition} Nakayama introduced in \cite{Nak} a divisorial decomposition of pseudo-effective divisors that generalizes the classical Zariski decomposition for divisors on surfaces. We denote by $D = P_\sigma (D) + N_\sigma (D)$ the Zariski decomposition of a divisor $D$. $P_\sigma(D)$ is a movable divisor and $N_\sigma(D)$ is the negative part and it is an effective divisor. We list in the following lemma the properties of the Zariski decomposition that we will need in the rest of the paper. Proofs can be found in \cite{Nak} and \cite{CHMS14}.
 
\begin{lemma}\label{zariski}
Let $D$ be a pseudo-effective divisor. Then 
\begin{itemize}
\item[i)] if $P \leq D$ and $P$ is movable then $P \leq P_\sigma(D)$,
\item[ii)] $\supp(N_\sigma(D))$ is the divisorial part of $\B_-(D)$, which we denote it by $\B^{div}_-(D)$.
\end{itemize}
\end{lemma}

The following two results from \cite{FKL} are of fundamental importance to this paper. They should be thought as a geometric characterization of the vanishing of the derivate of the volume function. Here we report only parts of the main theorems in \cite{FKL} which are the most relevant for us. 

\begin{theorem}\label{minus}
Let $D$ be a big divisor on $X$ and $E$ be an effective divisor on $X$. Then the following are equivalent:
\begin{enumerate}
\item[i)] $\vol(D-E) = \vol(D)$.
\item[ii)] $E \leq N_\sigma(D)$.
\item[iii)] $h^0(mD-mE) = h^0(mD)$ for all $m>0$.
\end{enumerate}
\end{theorem}

\begin{theorem}\label{plus}
Let $D$ be a big divisor on $X$ and $E$ be an effective divisor on $X$. Then the following are equivalent:
\begin{enumerate}
\item[i)] $\vol(D+E) = \vol(D)$.
\item[ii)] $\supp(E) \subseteq \B_+(D)$.
\item[iii)] $h^0(mD+rE) = h^0(mD)$ for all $m>0$ and $r>0$.
\end{enumerate}
\end{theorem}

\subsection{Numerical dimension revised}

We expect that the volume function can be used to compute the numerical dimension of any divisor $D$. Here we collect the results that will be needed later. The following is a special case of a result claimed in \cite{Leh} but in light of the example in \cite{Lei}, we provide an alternative proof. See \cite{McC} for a similar argument in the K\"ahler case. 

\begin{proposition}\label{lehmann}
Let $D$ be an effective divisor and let $A$ be a sufficiently ample divisor. Then $D = N_\sigma(D)$ if and only if there exists $C>0$ such that $\vol(mD+A) \leq C$ for all $m\geq0$. 
\end{proposition}

\begin{proof}
Suppose that $D = N_\sigma(D)$. Then 
\[
\supp(D) = \B^{div}_-(D) = \cup_m \B^{div}_+ \left( D + (1/m) A \right).
\]
As the sets $\B_+(D + (1/m) A)$ are increasing, we eventually get that $\supp(D) \subseteq \B_+(mD + A)$ for some $m$. Then the results follows from Theorem \ref{plus} part $i)$.

Suppose that $\vol(mD+A) \leq C$. As $D$ is effective, we can assume that $\vol(mD+A)$ is constant in $m$. In particular, Theorem \ref{minus} implies that $P_\sigma (D) \leq N_\sigma(mD+A)$ for some $m$. This immediately implies that $P_\sigma(D) = 0$ using Lemma \ref{zariski}.
\end{proof}

We will also need a slightly different statement. 

\begin{proposition}\label{lehmann2}
Let $D$ be an effective divisor. Then $D = N_\sigma(D)$ if and only if $h^0(mD+A) \leq C$ for all $m\geq 0$. 
\end{proposition}

\begin{proof}
Suppose $D = N_\sigma(D)$. Then, as before, $\supp(D) \subseteq \B_+(mD + A)$ and we conclude using Theorem part $iii)$.

Suppose that $h^0(mD+A) \leq C$. Then the result follows form \cite{Nak} and \cite{CHMS14} in positive characteristic.
\end{proof}

\section{Global sections on open varieties}
Let $X$ be a projective variety and let $E$ be a reduced divisor. For simplicity, we assume in this section that $X$ is a smooth variety. Let $D$ be a divisor on $X$. We are interested in global sections of $mD$ in the open set $U:=X\backslash E$. 

Recall that for any $m>0$, we have
\bdism
H^{0}(U,\oo_{U}(mD|_{U})=\lim_{k \to \infty}H^{0}(X,\oo_{X}(mD+kE)).
\edism
In particular, by finite dimensionality, we have that there exists a function $k(m)$ of $m$ such that 
\bdism
H^{0}(U,\oo_{U}(mD|_{U})=H^{0}(X,\oo_{X}(mD+k(m)E). 
\edism
When no confusion arises we denote $k(m)$ by $k$. Let $s\in H^{0}(U,\oo_{U}(mD|_{U})$ and $\bar{s}$ be the corresponding section in  $H^{0}(X,\oo_{X}(mD+kE)$. A priori, $k$ can be an arbitrary function of $m$ but, if $D$ is big, we will show that there exists $a$ such $k\leq ma$. In particular, the order of the pole of $\bar{s}$ along $E$ is at most linear in $m$.

We will prove our main results in few steps. The first one is to show that $E$ must have numerical dimension $0$.

\begin{lemma}\label{numeric}
Let $D$ be a big divisor. If $H^{0}(U,\oo_{U}(mD|_{U})$ is finite for $m$ sufficiently large then $\K(E)=0$.
\end{lemma}

\begin{proof}
Assume by contradiction that $\nu=\K(E)>0$. By Theorem \ref{lehmann}, for $A$ sufficiently ample there exists $C>0$ such that   
\[
h^{0}(X,\oo_{X}(kE+A))\geq Ck^{\nu},
\]
for $k$ sufficiently big. Since $D$ is big, using Kodaira's lemma \cite{Laz1}, we can fin $m$ big enough such that $mD-A$ is effective. In particular, 
\[
h^{0}(X,\oo_{X}(mD+k E))\geq h^{0}(X,\oo_{X}(A+kE))\geq Ck^{\nu}.
\] 
This contradicts the finiteness of the direct limit. 
\end{proof}

For the convenience of the reader, we slightly rephrase here a part of Theorem \ref{plus}.

\begin{proposition}\label{contained}
Let $X$ be a projective variety. Let $D$ be a big divisor and let $E$ be an effective divisor. Suppose that there exists $a\geq 0$ such that $\supp(E)\subseteq \B_+(D+aE)$. Then
\[
h^{0}(X,\oo_{X}(mD+kE))=h^{0}(X,\oo_{X}(m(D+aE)),
\]  
for any $m>0$ and any $k \geq ma$.
\end{proposition}

On the other hand, if $\supp(E)$ is not contained in the augmented base locus then finite dimensionality of global sections will fail.

\begin{proposition}\label{notcontained}
Let $X$ be a projective variety. Let $D$ be a big divisor and let $E$ be a reduced divisor. Suppose that for any $a\gg 0$ we have that $\supp(E)\nsubseteq \B_+(D+aE)$. Then $\K(E)>0$.
\end{proposition}

\begin{proof}
We first assume that $E$ is irreducible and reduced. By assumption, there exists a positive integer $a_0$ such that $\supp(E) \nsubseteq \B_+(D+aE)$ for any $a\geq a_0$. Replacing $D$ with $D+a_0 E$, we can assume that $a_0 = 0$. By the characterization of the augmented base locus, we have that $\vol_{X|E}(D+aE)>0$ for any $a\geq 0$. Then for any $a \geq 0$, we have
\begin{align} \notag
\vol_{X|E}(D+aE)&=\vol_{X|E}\left(\frac{1}{2}D+aE +\frac{1}{2}D\right) \\ \notag
&\geq \vol_{X|E}\left(\frac{1}{2}D\right)=\frac{1}{2^{n-1}}\vol_{X|E}(D)=\epsilon>0,
\end{align}
where the inequality is obtained using the log-concavity of the restricted volume function, see \cite{ELMNP} for the original result and Corollary 4.20 in \cite{LM} for arbitrary characteristic. 

By Theorem \ref{okounkov}
\[
\vol(D+aE)\geq n \int_{0}^{a}\vol_{X|E}(D+tE) dt + \vol(D)\geq \epsilon a +\vol(D).
\]
In particular, there exists a constant $C>0$ such that $\vol(D+aE)>Ca$ for any $a>0$. Fix $A$ sufficiently ample such that $A-D$ is effective. Then $\vol(A+aE)>Ca$ for any $a>0$. This contradicts Proposition \ref{lehmann}.

For a general reduced divisor, the same argument works once we replace the restricted volume with the positive product defined in \cite{Cut}. Then Theorem 5.6 and Theorem 6.7 in there do the job.
\end{proof}

We can now prove the main result of this section, which implies the theorems in the introduction. 

\begin{theorem}\label{augtheo}
Let $X$ be a projective variety and let $E$ be a reduced divisor. Write $U:=X\backslash E$. Let $D$ be a big divisor on $X$. Then $H^{0}(U,\oo_{U}(m D|_{U}))$ is finite dimensional for any $m>0$ if and only if there exists $a\geq 0$ such that $\supp(E)\subseteq \B_{+}(D+aE)$.
\end{theorem} 

\begin{proof}
If $H^{0}(U,\oo_{U}(m D|_{U}))$ is finite dimensional for any $m>0$ then by Lemma \ref{numeric}, $\K(E) = 0$. Since the numerical dimension is subadditive, the same is true for any irreducible components of $E$.  Then there must be $a$ such that $\supp (E)\subseteq \B_{+}(D+aE)$, otherwise we would get a contradiction using Proposition \ref{notcontained}. 

If there exists $a$ such that $\supp(E)\subseteq \B_{+}(D+aE)$, then by Proposition \ref{contained}
\[
h^0(U,\oo_{U}(m D|_{U})) = h^{0}(X,\oo_{X}(mD+kE)) = h^{0}(X,\oo_{X}(m(D+aE)),
\]
for any $m>0$. This finishes the proof of the theorem.
\end{proof}

\section{Proof of the Theorems}

\begin{proof}[Proof of Theorem \ref{bigtheo}]

As shown in Proposition 2.1 in \cite{FKL}, we can assume that $X$ is smooth and we can use the results in the previous section.

If $E\leq N_\sigma (D+aE)$, then 
\[
\supp(E) \subseteq \supp N_\sigma (D+aE) = \B^{div}_- (D+aE) \subseteq \B_+ (D+aE).
\]
We can apply directly Theorem \ref{augtheo} to conclude. 

Assume that the cohomology groups in questions are finite dimensional. Recall that by Theorem \ref{augtheo} and Proposition \ref{contained}, we have that 
\[
H^{0}(X,\oo_{X}(m(D+aE)) = H^{0}(X,\oo_{X}(m(D+(a+1)E)),
\]
for any $m \geq 0$. In particular, 
\begin{align}\notag
H^{0}(X,\oo_{X}(m(D+(a+1)E)) &= H^{0}(X,\oo_{X}(m(D+aE)) \\ \notag
&= H^{0}(X,\oo_{X}(m(D+(a+1)E)-mE)).
\end{align}
Then Theorem \ref{minus} and Theorem \ref{augtheo} imply the desired result.
\end{proof}

Finally, we can prove the main theorem of this paper. 

\begin{proof}[Proof of Theorem \ref{main}]

By the same argument in Corollary 2.1.38 in \cite{Laz1}, we can assume that $D$ is big. It is a standard argument but for convenience of the reader we sketch the ideas here. 

Consider the collection of finite dimensional vector spaces $V_m := H^0(U,\oo_U(mD|_U))$. The section ring is the graded algebra
\[
R(V_\bullet) := \bigoplus_{m=0}^{\infty}  V_m.
\]
Let $\phi_m : U \subset X \dashrightarrow Y_m \subset \pp (V_m)$ be the map defined by evaluating global sections. The main theorem of \cite{CJ19} shows the existence of the Iitaka fibration associated to the algebra $R(V_\bullet)$. In particular, there exists a fixed morphism $\phi_{\infty}: X_{\infty} \rightarrow Y_{\infty}$ such that the rational maps $\phi_m$ are birationally equivalent to $\phi_\infty$, for $m$ large enough.

Since the assertion is invariant under birational transformations, we may replace $X$ with $X_\infty$. Moreover, we can assume that $Y_\infty$ is smooth. For simplicity, let us denote by $\phi: X \rightarrow Y$ the Iitaka fibration and tet $V$ be the image of $U$ in $Y$.  As in the proof of Corollary 2.1.38 in \cite{Laz1}, we can assume $D$ is vertical for $\phi$ and then there exists an ample divisor $H$ on $Y$ such that
\[
h^0(U,\oo_U(mD|_U)) \leq h^0(V,\oo_{V}(mH|_V)).
\]
We need to show that $h^0(V,\oo_{V}(mH|_V))$ if finite dimensional. To show that it suffices to show that there exists $m_0$ such that $\oo_X(m_0 D) \otimes \phi^*\oo_Y(- H)$ has a global section. This is done in Theorem 2.1.33 in \cite{Laz1}. In particular, for any $m>0$
\[
h^0(V,\oo_{V}(mH|_V)) \leq h^0(U,\oo_U(m_0mD|_U)).
\]

This shows that the asymptotic behavior of the sections of $D$ can be estimated using a big divisor on the base of the Iitaka fibration. In particular, we can assume $D$ to be a big divisor with finite dimensional $h^0(U,\oo_U(mD|_U))$ for any $m$.

By Theorem \ref{bigtheo} and Proposition \ref{contained}, we have that there exists $a$ such that
\[
h^0(U,\oo_U(mD|_U)) = h^0(X,\oo_X(m(D+a E )),
\]
for any $m\geq 0$. Then the result follows from the classical result, Theorem \ref{old}.
\end{proof}

\section{Further remarks}

We conclude this paper collecting some results on the finiteness of the cohomology groups considered before. Theorem \ref{bigtheo} deals with the finiteness in the case $D$ is a big divisor. It would be interesting to understand the general case as well. It is not hard to do it in the case of surfaces.  

First of all, the same argument of Lemma \ref{numeric} implies that we must have $\kappa(E) = 0$. Using Riemann-Roch we can conclude a stronger statement.

\begin{proposition}
Let $X$ be a smooth surface and let $E$ be a reduced divisor with $E^{2}\geq 0$. Let $D$ be a pseudo-effective divisor such that $D\cdot E\geq 1$ if $E^{2}=0$, there exists $m_{0}$ such that $H^{0}(U,\oo_{U}(mD|_{U}))$ is infinite dimensional for any $m\geq m_{0}$. 
\end{proposition}
\begin{proof}
Recall that for any $k\geq 0$ we have a natural inclusion 
\bdism
H^{0}(X,\oo_{X}(mD+kE))\hookrightarrow H^{0}(U,\oo_{U}(mD|_{U})),
\edism
so it suffices to show that there exists $m_{0}$ such that
\bdism
\lim_{k\to\infty}h^{0}(X,\oo_{X}(mD+kE))=\infty,
\edism
for any $m\geq m_{0}$. By Riemann-Roch
\begin{align}\notag
h^{0}(X,\oo_{X}(mD+kE))+h^{2}(X,\oo_{X}(mD+kE))\geq  \\ \notag
\frac{(mD+kE)^{2}}{2}-m\frac{D\cdot K_{X}}{2}-k\frac{E\cdot K_{X}}{2}+\chi(\oo_{X}).
\end{align}
Moreover, by Serre duality 
\bdism
h^{2}(X,\oo_{X}(mD+kE))=h^{0}(X,\oo_{X}(K_{X}-mD-kE))\leq h^{0}(X,\oo_{X}(K_{X})).
\edism
All together
\begin{align}\notag
h^{0}(X,\oo_{X}(mD+kE))&\geq k^{2}\frac{E^{2}}{2}+k \left(mD\cdot E-\frac{E\cdot K_{X}}{2}\right) \\ \notag
&+\frac{(mD)^{2}-mD\cdot K_{X}}{2}+\chi(\oo_{X})-h^{0}(X,\oo_{X}(K_{X})).
\end{align}
In particular, if $E^{2}>0$ the limit is infinite for any $m\geq 1$. If $E^{2}=0$, then by assumption, we can find $m_{0}$ such that $m_{0}D\cdot E-E\cdot K_{X}/2 >0$. The limit is again infinite for $m\geq m_{0}$.
\end{proof}

The above result imposes some extra conditions on the positivity of $E$, but they are still weaker than the case when $D$ is big where $\K(E)$ essentially characterizes the finite dimensionality of the spaces of global sections. When $D$ is not big, there are exceptions to Theorem \ref{bigtheo}. For example, if $\kappa(E) = 0$, $\K(E) =1 $ and $D = E$, we have new examples where finite dimensionality of global sections holds. The next theorem says that these are basically the only exceptions. 

\begin{theorem}
Let $X$ be a normal projective surface. Let $D$ be a pseudo-effective Cartier divisor and let $E$ be a reduced divisor. Then $h^{0}(U,\oo_{U}(mD|_{U}))$ is finite dimensional for all $m\geq 0$ if and only if one of the following holds: 
\begin{itemize}
\item[i)] $E \leq N_\sigma(D + a E)$ for some $a$, or
\item[ii)] $\kappa(D)=0$ and $P_\sigma(D) \equiv a P_\sigma(E)$ for some rational number $a$, or
\item[iii)] $D \equiv N_\sigma(D)$ and $\kappa(E) =0$.
\end{itemize}
\end{theorem}

\begin{proof}

If one of the three conditions holds, it is easy to show that the groups in question are finite dimensional. 

For the other direction, let us denote by $E = P + N$ the Zariski decomposition of $E$. First note, If $P$ is numerically trivial, then there is nothing to prove as it implies that $E = N$ and condition $i)$ is satisfied. So we can assume that $P$ is not numerically trivial. 

For the sake of simplicity, let us assume at the moment that $D$ is a non-zero nef divisor. Since $N$ is effective, we have that for any $m\geq 0$ and $k\geq 0$
\[
H^0(X, \oo(mD + kP )) \hookrightarrow H^0(X, \oo(mD + kE )).
\]

As $P$ is a nef divisor, we can assume that $P^2 =0$ otherwise the groups $H^0(X, \oo(mD + kE ))$ will grow unbounded in $k$ by Lemma \ref{zariski}. Moreover, by the same result, we can assume that $D\cdot P =0 $. By Hodge index theorem we get that $D\equiv a P$ for some rational number $a$. 

The above argument implies that if $D = P_\sigma(D) + N_\sigma(D)$ is the Zariski decomposition of $D$ and $P_\sigma(D)$ is not trivial, then condition $ii)$ is satisfied. 

Finally if $P_\sigma(D)$ is trivial. then the finite dimensionality of the groups implies just that the Kodaira dimension of $E$ is $0$. 
\end{proof}

It would be interesting to carry out a similar analysis for higher dimensional varieties.

\begin{question}
Let $D$ be a pseudo-effective divisor with $\K(D) < n$. When is $H^{0}(U,\oo_{U}(mD|_{U}))$ a finite dimensional vector space for any $m$ large enough? 
\end{question}

\end{document}